\documentclass[11pt]{article}

\usepackage{fullpage}
\usepackage[T1]{fontenc}
\usepackage{amsmath,amssymb}
\usepackage{amsopn,amsthm}

\usepackage{graphicx}
\usepackage{color}
\usepackage{psfrag}
\usepackage{epsfig}
\usepackage{epstopdf}
\usepackage{tikz}

\usepackage{mathrsfs}
\usepackage{bm}
\usepackage{cases}

\usepackage{algpseudocode}
\usepackage{algorithm}
\floatstyle{plain} \newfloat{myalgo}{tbhp}{mya}

{\begin{myalgo}[#1]
    \centering
    \begin{minipage}{0.9\textwidth}
      \begin{algorithm}[H]}%
      {\end{algorithm}
    \end{minipage}
  \end{myalgo}}

\usepackage[colorlinks=true]{hyperref}
\hypersetup{urlcolor=blue, citecolor=blue, linkcolor=blue}

\usepackage{enumitem}

\newtheorem{theorem}{Theorem}[section]

\newtheorem{lemma}[theorem]{Lemma}

\theoremstyle{definition}

\newtheorem{example}[theorem]{Example}

\theoremstyle{remark} \newtheorem{remark}[theorem]{Remark}

\numberwithin{equation}{section}
\numberwithin{figure}{section}
\numberwithin{algorithm}{section}

\newcommand{\field}[1]{{\mathbb{#1}}}
\newcommand{\C}{\field{C}}
\newcommand{\N}{\field{N}}
\newcommand{\R}{\field{R}}

\newcommand{\Acal}{\mathcal{A}}
\newcommand{\Bcal}{\mathcal{B}}

\newcommand{\Lcal}{\mathcal{L}}

\newcommand{\Ocal}{\mathcal{O}}

\newcommand{\Xcal}{\mathcal{X}}
\newcommand{\Ycal}{\mathcal{Y}}

\newcommand{\bs}{\boldsymbol}

\newcommand{\bfE}{{\bs E}}
\newcommand{\bfF}{{\bs F}}

\newcommand{\bfI}{{\bs I}}

\newcommand{\bfR}{{\bs R}}
\newcommand{\bfS}{{\bs S}}

\newcommand{\bfV}{{\bs V}}

\newcommand{\bfX}{{\bs X}}

\newcommand{\loc}{{\mathrm{loc}}}
\newcommand{\ol}[1]{\overline{#1}}

\newcommand{\subs}{\subseteq} 
\newcommand{\trans}{{\top}}

\newcommand{\ds}{\, \dif s}

\newcommand{\dx}{\, \dif x}
\newcommand{\dy}{\, \dif y}

\newcommand{\ahat}{\widehat a}

\newcommand{\hhat}{\widehat h}

\newcommand{\xhat}{\widehat{x}}

\newcommand{\dtilde}{{\widetilde d}}

\newcommand{\htilde}{{\widetilde h}}

\newcommand{\rmi}{\mathrm{i}} 

\newcommand{\rme}{\mathrm{e}}

\newcommand{\Sd}{{S^{d-1}}}
\newcommand{\Sone}{{S^1}}

\newcommand{\Rd}{{\R^d}}
\newcommand{\Rtwo}{{\R^2}}

\newcommand{\BR}{{B_R(0)}}

\newcommand{\ui}{u^i}

\newcommand{\uinfty}{u^\infty}

\newcommand{\uig}{u^i_g}
\newcommand{\uqgs}{u_{q,g}^s}
\newcommand{\uqg}{u_{q,g}}
\newcommand{\uqginfty}{u_{q,g}^\infty}

\newcommand{\LtSd}{{L^2(S^{d-1})}}

\newcommand{\di}{\partial}

\newcommand{\qmin}{q_\mathrm{min}}
\newcommand{\qmax}{q_\mathrm{max}}

\DeclareMathAlphabet{\mathbi}{\encodingdefault}{\rmdefault}{\bfdefault}{\itdefault}
\DeclareRobustCommand{\vec}[1]{\ifmmode\mathbi{#1}\else\textbf{\textit{#1}}\fi}

\DeclareMathOperator{\dif}{d\!}

\DeclareMathOperator{\supp}{supp}

\DeclareMathOperator{\sinc}{sinc}
\DeclareMathOperator{\trace}{trace}

\DeclareMathOperator{\real}{Re}

\DeclareMathOperator*{\essinf}{ess\,inf}

\begin{document}

\title{Monotonicity-based regularization of inverse medium scattering for shape reconstruction}
\author{Roland Griesmaier\footnote{Institut f\"ur Angewandte und Numerische
    Mathematik, Karlsruher Institut f\"ur Technologie, Karlsruhe,
    Germany ({\tt roland.griesmaier@kit.edu})}\,\! ,
  Bastian Harrach\footnote{Institute for Mathematics,
    Goethe-University Frankfurt, Frankfurt am Main, Germany
    ({\tt harrach@math.uni-frankfurt.de})}\,\! , and
  Jianli Xiang\footnote{Corresponding author. Three Gorges Mathematical Research Center,
    College of Mathematics and Physics, China Three Gorges University, Yichang 443002, China
    ({\tt xiangjianli@ctgu.edu.cn})}\,\! .
}
\date{\today}

\maketitle

\begin{abstract}
    We consider the scattering of time-harmonic plane waves by a compactly supported inhomogeneous scattering obstacle governed by the Helmholtz equation.
    Given far field observations of the scattered fields corresponding to plane wave incident fields for all possible incident and observation directions we study the inverse problem to recover the support of the scatterer.
    We propose a qualitative monotonicity-based regularization scheme which combines monotonicity-based shape reconstruction with one-step linearization to reconstruct a discrete approximation of the shape of the scatterer from noisy far field data.
    The purpose of the one-step linearization is to stabilize the monotonicity approach to shape reconstruction.
    We show that the monotonicity-based regularization scheme recovers the correct shape of the scatterer for noise-free data.
    Furthermore, we establish that the solution of the monotonicity-based regularization converges to the exact solution as the noise level tends to zero.
    We present numerical examples to illustrate our theoretical findings.
\end{abstract}

{\small\noindent
  Mathematics subject classifications (MSC2020): 35R30, (78A45)
  \\\noindent
  Keywords: inverse scattering, Helmholtz equation, shape reconstruction, monotonicity, regularization, global convergence
  \\\noindent
  Short title: Monotonicity based regularization
}

\section{Introduction}
\label{sec:Introduction}
The determination of a compactly supported scattering object given knowledge of far field observations of incident waves and scattered waves is a severely ill-posed nonlinear inverse problem that has received considerable attention in the last fifty years (see, e.g., \cite{ColKre19}).
In this work the wave motion will be described by a scalar Helmholtz equation.
We consider inhomogeneous scatterers, where the incident wave fields penetrate into the interior of the obstacle, and as usual we make the assumption that the scatterer is embedded in an unbounded uniform medium.
In this case, the reconstruction methods available in the literature can basically be divided into three different classes.
These are
(i)~methods based on linearized inverse scattering models such as the Born approximation
(see, e.g.,~\cite{AGKLS12,GuzBon06,KilMosSch12} or~\cite[Sec.~8.1]{Dev12}),
(ii)~iterative regularization schemes based on nonlinear optimization
(see, e.g.,~\cite{DorLes06,GutKli94,Het94,HohWei15,KleBer92}), and
(iii)~qualitative shape reconstruction schemes based on some sort of sampling strategy
(see, e.g.,~\cite{AudHad14,ColMon88,ColPotPia97,GriHar18,ItoJinZou12,Kir02}).
The first class of methods usually only works well at low frequencies or for small contrasts in the refractive index between the background medium and the scattering obstacle, when linearizations of the scattering problem give suitable approximations.
Iterative regularization algorithms require the solution of many direct scattering problems during the reconstruction, which is computationally expensive, and they often also suffer from possible local minima.
Qualitative sampling-type reconstruction methods usually come with a sound mathematical justification, and they are computationally efficient since no nonlinear optimization methods are required for their resolution.
They are only mildly dependent on the geometry and physical properties of the scatterer, but on the other hand most qualitative reconstruction methods are rather sensitive to noise in the given data.

The purpose of this paper is to initiate the study of reconstruction schemes for inverse medium scattering problems that combine the rigorous theoretical foundation of qualitative reconstruction methods with the favorable stability properties of Born-based inverse scattering or iterative regularization schemes.
A monotonicity-based regularization method for solving the inverse conductivity problem by combining the qualitative sampling-based monotonicity method from~\cite{HarUlr13,Tam06,TamRub02} with a one-step linearization (see, e.g., \cite{CINSG90}) has been proposed in~\cite{HarMin16,HarMin18}.
This technique has later been applied to an inverse boundary value problem in stationary elasticity in~\cite{EbeHar22}.
Recently, the first extension of this framework to an inverse boundary value problem for the Helmholtz equation has been established in~\cite{EbeHarWan25}.
This work builds on the monotonicity-based reconstruction method from \cite{HarPohSal19b,HarPohSal19a}.
The non-coercivity of the Helmholtz equation renders this generalization significantly non-trivial and requires a number of new ideas.

Following~\cite{EbeHarWan25}, we develop a constrained optimization problem for the sum of the eigenvalues of the linearized residual operator associated to the inverse scattering problem.
Here the admissible set of refractive indices is determined by a monotonicity criterion from~\cite{GriHar18}, which is based on a monotonicity relation between the index of refraction of the scatterer and the eigenvalues of the linearized residual operator.
We show that in the noise-free case the unique minimizer of this constrained minimization problem exactly determines the support of the scatterer.
This is mainly due to the fact that the minimizer coincides with the largest element in the admissible set, which is determined by the monotonicity test.
Therefore, the theoretical results for the monotonicity method from \cite{GriHar18} can be applied.
For noisy data we have to suitably modify the cost functional as well as the admissible set in the constrained optimization problem.
The admissible set is defined in terms of a relaxed monotonicity criterion (see, e.g., \cite{GarSta17,GarSta19} for contributions on the regularization of the monotonicity method for the inverse conductivity problem).
In the cost functional we add a scalar multiple of the norm of the linearized residual operator as a regularization term to the sum of the eigenvalues of the linearized residual operator.
This significantly improves the reconstruction in particular at low frequencies or for small contrasts in the refractive index, when linearizing of the forward problem is appropriate.
Our main theoretical result for noisy data states that the reconstructions of the support of the scatterer obtained by minimizing the regularized cost functional tend to the exact support of the scatterer as the noise level tends to zero.

Since the constrained optimization problem associated to the new monotonicity-based regularization scheme is equivalent to a convex semidefinite minimization problem, it can be implemented efficiently using standard software packages.
In our numerical examples we use CVX from~\cite{CVX1,CVX2}, and we compare the reconstructions obtained by this new approach to results obtained by the factorization method from~\cite{Kir02}.

The remainder of this article proceeds as follows.
In Section~\ref{sec:Setting} we briefly outline the theoretical background on the direct scattering problem and introduce the far field operator and its linearization.
Then, in Section~\ref{sec:MonotonicityRegularization}, we develop the monotonicity-based regularization method for the inverse scattering problem, and we analyze convergence and stability of this approach.
Numerical results are given in Section~\ref{sec:Numerics}, and we close with some conclusions.

\section{Scattering by an inhomogeneous obstacle}
\label{sec:Setting}
We consider scattering of time-harmonic scalar waves in $\Rd$, $d=2,3$, by an inhomogeneous obstacle supported in some bounded open subset $D\subs\Rd$.
Let $\uig$ be a Herglotz incident field with density $g\in\LtSd$ and wave number $k>0$, which is defined by
\begin{equation}
    \label{eq:Uig}
    \uig(x)
    \,:=\, \int_{\Sd} \rme^{\rmi k x\cdot \theta} g(\theta) \ds(\theta) \,, \qquad x\in\Rd \,.
\end{equation}
The scattering object is fully described by the refractive index $n^2=1+q\in L^\infty(\Rd)$, where ${q\in L^\infty(D)}$ denotes a contrast function.
Throughout, we identify functions in $L^\infty(D)$ with their extension by zero to all of $\Rd$ whenever appropriate.
We also assume that
\begin{equation*}
    0<\qmin\leq q\leq \qmax<\infty \qquad \text{a.e.\ in } D
\end{equation*}
for some constants $\qmin,\qmax>0$.
Then, the total field $\uqg\in H^1_\loc(\Rd)$ is the unique weak solution of the Helmholtz equation
\begin{equation*}
    \Delta\uqg + k^2 n^2 \uqg = 0 \qquad \text{in } \Rd
\end{equation*}
such that the associated scattered field, $\uqgs := \uqg-\uig$, satisfies the Sommerfeld radiation condition
\begin{equation}
    \label{eq:Sommerfeld}
    \lim_{r\to\infty} r^{\frac{d-1}{2}}
    \Bigl( \frac{\di\uqgs}{\di r}(x) - \rmi k \uqgs(x) \Bigr)
    \,=\, 0 \,, \qquad r=|x| \,,
\end{equation}
uniformly with respect to all directions $x/|x|\in\Sd$.
Accordingly, the scattered field has the asymptotic behavior of an outgoing spherical wave
\begin{equation*}
    \uqgs(x)
    \,=\, C_d \frac{\rme^{\rmi k |x|}}{|x|^{\frac{d-1}{2}}} \uqginfty(\xhat)
    + \Ocal\bigl(|x|^{-\frac{d+1}{2}}\bigr) \,, \qquad |x|\to\infty \,,
\end{equation*}
uniformly in $\xhat:=x/|x|\in\Sd$, where
\begin{equation*}
    C_d \,=\, \rme^{\rmi\pi/4} / \sqrt{8\pi k} \quad\text{if } d=2
    \qquad\text{and}\qquad
    C_d \,=\, 1 / (4\pi) \quad\text{if } d=3 \,.
\end{equation*}
The function
\begin{equation}
\label{eq:Farfield}
    \uqginfty(\xhat)
    \,=\, k^2 \int_{D} q(y) \uqg(y) \rme^{-\rmi k\xhat\cdot y} \dy \,, \qquad \xhat\in\Sd \,,
\end{equation}
is known as the far field pattern of $\uqgs$ (see, e.g., \cite[Chap.~7]{Kir21} for further details on the direct scattering problem).

Next, we define the far field operator
\begin{equation}
    \label{eq:FarfieldOperator}
    F[q]:\, \LtSd \to\LtSd \,, \quad F[q]g \,:=\, \uqginfty \,,
\end{equation}
which maps densities of superpositions of plane wave incident fields to the far field patterns of the associated scattered fields.
As such it is often used as an idealized mathematical model for remote sensing experiments.
We recall that the far field operator is compact and normal; see, e.g., \cite[Thm.~7.20]{Kir21}.
The inverse scattering problem we are concerned with is to recover the support $D$ of the scatterer from finitely many observations of $F[q]$.
Here, we assume that we are also given some sufficiently large region of interest $\Omega\subs\Rd$ that contains the scatterers in its interior.
We note that at least the size of a sufficiently large ball or cube containing all scatterers can immediately be estimated from the decay behavior of the Fourier or spherical harmonics coefficients of the given far field data; see, e.g., \cite{KusSyl03,Syl06}.

Denoting by $\Xcal\subs L^\infty(\Rd)$ the subspace of functions $q\in L^\infty(\Rd)$ with $\supp(q) \subs \Omega$ and $\essinf_{x\in\supp(q)} q(x) > 0$, we write $F: \Xcal \to \Lcal(\LtSd)$ for the nonlinear operator that maps a contrast function $q\in\Xcal$ to the associated far field operator $F[q]\in\Lcal(\LtSd)$.

\begin{lemma}
\label{lmm:FrechetDerivative}
   The nonlinear operator $F$ is Fr\'echet differentiable at zero, and for any compactly supported $h\in L^\infty(\Omega)$ its Fr\'echet derivative satisfies
   \begin{equation}
   \label{eq:FrechetDerivative}
   \int_{\Sd} g \ol{(F'[0] h)g} \ds
   \,=\, k^2 \int_\Omega h |\uig|^2 \dx
   \qquad \text{for all } g\in\LtSd \,.
   \end{equation}
\end{lemma}

\begin{proof}
    For $h\in L^\infty(\Omega)$ compactly supported and $g\in \LtSd$ let
    \begin{equation*}
        \bigl((F'[0]h)g\bigr)(\xhat)
        \,:=\, k^2 \int_\Omega h(y) \uig(y) \rme^{-\rmi k\xhat\cdot y} \dy \,.
    \end{equation*}
    Using \eqref{eq:FarfieldOperator} and \eqref{eq:Farfield} we find that
    \begin{equation}
    \label{eq:ProofFrechetDerivative1}
    \begin{split}
        \| F[h]g - F[0]g - (F'[0]h)g \|_{\LtSd}^2
        &\,=\, \int_\Sd \biggl|
        k^2 \int_\Omega h(y) \bigl( u_{h,g}(y) - \uig(y) \bigr) \rme^{-\rmi k\xhat\cdot y} \dy
        \biggr|^2 \ds(\xhat) \,.
    \end{split}
    \end{equation}
    Since, $u^s_{h,g} := u_{h,g} - \uig \in H^1_\loc(\Rd)$ satisfies
    \begin{equation*}
        \Delta u^s_{h,g} + k^2(1+h) u^s_{h,g}
        \,=\, - k^2 h \uig \qquad \text{in } \Rd \,,
    \end{equation*}
    together with the Sommerfeld radiation condition \eqref{eq:Sommerfeld}, we obtain using the well-posedness of this scattering problem or, equivalently, the invertibility of the associated Lippmann-Schwinger integral equation (see, e.g., \cite[Thm.~7.13]{Kir21}) that
    \begin{equation*}
        \begin{split}
            \biggl| k^2 \int_\Omega h(y) u^s_{h,g}(y) \rme^{-\rmi k\xhat\cdot y} \dy \biggr|
            &\,\leq\, k^2 \|h\|_{L^\infty(\Omega)} |\Omega|^{1/2} \|u^s_{h,g}\|_{L^2(\Omega)}
            \,\leq\, k^2 \|h\|_{L^\infty(\Omega)} |\Omega|^{1/2} \|k^2 h\uig\|_{L^2(\Omega)} \\
            &\,\leq\, k^4 |\Sd|^{1/2} |\Omega| \|h\|_{L^\infty(\Omega)}^2 \|g\|_{\LtSd} \,.
        \end{split}
    \end{equation*}
    Substituting this estimate into \eqref{eq:ProofFrechetDerivative1} shows that the operator
    $F'[0]: L^\infty(\Omega) \to \Lcal(\LtSd)$ defined by \eqref{eq:ProofFrechetDerivative1} is the Fr\'echet derivative of the nonlinear operator $F$ at zero.
    The quadratic representation \eqref{eq:FrechetDerivative} then follows from~\eqref{eq:Uig}.
\end{proof}

We note that $F'[0]h$ is just the Born approximation of $F_h$; see, e.g., \cite{Kir17}.
It can immediately be seen that $F'[0]h$ is compact and self-adjoint.

\section{Monotonicity-based regularization of one-step linearization}
\label{sec:MonotonicityRegularization}
In this section we develop the mathematical framework of the novel monotonicity-based regularization for the inverse scattering problem, which recovers an approximation of the support $D$ of the scatterer from finitely many observations of the far field patterns of scattered fields corresponding to finitely many Herglotz incident fields.
The method combines the sound theoretical foundation of the monotonicity method with the numerical stability of a one-step linearization.

Denoting by $\{g_1,g_2,\ldots,g_N\}\subs\LtSd$ an orthonormal system in $\LtSd$, we assume that we have access to finitely many observations
\begin{equation}
\label{eq:DefFNq}
    \bfF_N[q]
    \,:=\, \Bigl(
    \int_\Sd g_j \ol{F[q] g_l} \ds
    \Bigr)_{j,l=1}^N
    \,\in\, \C^{N\times N} \,.
\end{equation}
Since the far field operator $F[q]$ is normal, the matrix $\bfF_N[q]$ is normal as well.
From Lemma~\ref{lmm:FrechetDerivative} we obtain that the Fr\'echet derivative $\bfF_N'[0]$ of $\bfF_N$ in zero is given by
\begin{equation}
\label{eq:DefFN'0}
    \bfF_N'[0] h
    \,=\, \Bigl(
    k^2 \int_\Rd h \ui_{g_j} \ol{\ui_{g_l}} \dx
    \Bigr)_{j,l=1}^N
    \,\in\, \C^{N\times N}
\end{equation}
for all compactly supported $h\in L^\infty(\Omega)$.
Since $h$ is real-valued, we see that the matrix $\bfF_N'[0] h$ is self-adjoint.

Our goal is to approximate the support of the scatterer $D$ from the finite set of observations described by $\bfF_N[q]$.
To solve this inverse problem, we discretize the region of interest $\ol{\Omega} = \bigcup_{m=1}^M \ol{P_m}$ into $M$ disjoint open pixels $P_m\subs\Omega$ and approximate the unknown contrast function $q$ by a piecewise constant function $h = \sum_{m=1}^M a_m \chi_{P_m}$ with unknown coefficients $a_1,\ldots a_M\in [0,\infty)$.
As usual $\chi_{P_m}$ denotes the indicator function for $P_m$.
We write
\begin{equation*}
    \Ycal
    \,:\,= \Bigl\{ h = \sum_{m=1}^M a_m \chi_{P_m} \;\Big|\;
    a_1,\ldots a_M\in [0,\infty) \Bigr\} \subs L^\infty(\Omega) \,,
\end{equation*}
and define the linearized residual ${\bfR: \Ycal \to \C^{N\times N}}$ by
\begin{equation*}
    \bfR (h)
    \,:=\, \real \bigl( \bfF_N[q] - \bfF_N'[0]h \bigr) \,.
\end{equation*}
Let $\lambda_1(\bfR(h)) \geq \lambda_2(\bfR(h)) \geq \cdots \geq \lambda_N(\bfR(h))$ denote the eigenvalues of this self-adjoint matrix in descending order.
To recover the support of a piecewise constant function $h$, which approximates the support $D$ of the scatterer, we consider the constrained minimization problem
\begin{subequations}
\label{eq:MinimizationProblem}
\begin{equation}
    \min_{h\in\Acal} \Bigl(
    \sum_{\{ j \;|\; \lambda_j > 0\}} \lambda_j(\bfR(h)) \Bigr) \,,
\end{equation}
where the admissible set $\Acal$ is given by
\begin{equation}
    \Acal
    \,:=\, \Bigl\{
    h = \sum_{m=1}^M a_m \chi_{P_m} \;\Big|\;
    0 \leq a_m \leq \min(\qmin,\beta_m)
    \Bigr\} \,.
\end{equation}
\end{subequations}
Here, the linear constraint $\beta_m$ is defined by a monotonicity test in terms of the infinite dimensional far field operator $F[q]$ and its Fr\'echet derivative $F'[0]$ with respect to $q$ at zero as follows:\footnote{Given any two compact self-adjoint operators $A$ and $B$ on a Hilbert space $X$, we write $A\geq_r B$ for some $r\in\N$ when $A-B$ has at most $r$ negative eigenvalues.}
\begin{equation}
\label{eq:Defbetam}
    \beta_m
    \,:=\,
    \max \bigl\{ \beta\geq 0 \;\big|\;
    \real ( F[q] - \beta F'[0]\chi_{P_m} ) \geq_{d_m} 0 \bigr\} \,,
\end{equation}
where $d_m = d(\qmin,P_m)$ is the number of negative eigenvalues of $\real (F[q]-\qmin F'[0]\chi_{P_m})$ in case this number is finite;
if $d(\qmin,P_m)$ is not finite, we set $\beta_m=0$.

The following lemma clarifies the connection between the linear constraint $\beta_m$ and the support of the scatterer $D$.
\begin{lemma}
\label{lmm:Pmbetam}
    Let $1\leq m\leq M$ and let $\beta_m\geq0$ be defined by \eqref{eq:Defbetam}.
    Then $P_m \subs D$ if and only if $\beta_m>0$.
    Moreover, in this case $\beta_m \geq \qmin$.
\end{lemma}

\begin{proof}
    We first show that $P_m\subs D$ implies $\beta_m \geq \qmin > 0$.
    To see this, let $P_m\subs D$ and choose $\beta\in [0,\qmin]$.
    Then, by the monotonicity relation \cite[Thm~5.1(a)]{GriHar18} and the non-negativity of the self-adjoint operator $F'[0]\chi_{P_m}$, there exists a finite dimensional subspace $V\subs\LtSd$ such that
    \begin{equation*}
        \begin{split}
            \real \Bigl(
            \int_\Sd& g \, \ol{(F[q]-\beta F'[0]\chi_{P_m}) g} \ds
            \Bigr) \\
            &\,=\, \real \Bigl(
            \int_\Sd g \, \ol{(F[q]-\qmin F'[0]\chi_{P_m}) g} \ds
            \Bigr)
            + \int_\Sd g \, \ol{(\qmin-\beta) (F'[0]\chi_{P_m}) g} \ds \\
            &\,\geq\, 0
        \end{split}
    \end{equation*}
    for all $g\in V^\perp$.
    In particular $d_m = d(\qmin,P_m) < \infty$ in the definition \eqref{eq:Defbetam} of $\beta_m$, and we can choose the subspace $V$ such that $\dim(V) \leq d_m = d(\qmin,P_m)$.
    Since $\beta\in[0,\qmin]$ was arbitrary, we obtain that $\beta_m\geq\qmin>0$.

    The fact that $P_m\not\subs D$ implies that $\beta=0$ follows immediately from the monotonicity relation~\cite[Thm.~5.1(b)]{GriHar18}, which says that in this case there exists no finite dimensional subspace $V\subs\LtSd$ such that~$\real ( F[q] - \beta F'[0]\chi_{P_m} ) \geq 0$ on $V^\perp$.
\end{proof}

\subsection{Convergence analysis for exact data}
\label{subsec:Convergence}

We show in Theorem~\ref{thm:MinimizerExactData} below that the constrained minimization problem \eqref{eq:MinimizationProblem} has a unique solution $\hhat\in\Acal$, and that $\supp(\hhat)$ approximates the support of the scatterer $D$.
For the proof of this result we require the following auxiliary lemma.

\begin{lemma}
\label{lmm:bfS_mPosDef}
    For all $m=1,\ldots,M$ the matrix $\bfF_N'[0]\chi_{P_m} \in \C^{N\times N}$ is self-adjoint and positive definite.
\end{lemma}

\begin{proof}
    Let $v = (v_1,v_2,\ldots,v_N)^\trans \in \C^N$, and let $\{g_1,g_2,\ldots,g_N\}\subs\LtSd$ be the orthonormal system introduced at the beginning of Section~\ref{sec:MonotonicityRegularization}.
    Using \eqref{eq:DefFN'0} we find that
    \begin{equation*}
        v^* \bigl( \bfF_N'[0]\chi_{P_m} \bigr) v
        \,=\, k^2 \sum_{j,l=1}^N \ol{v_j} v_l \int_\Rd \chi_{P_m} \ui_{g_j} \, \ol{\ui_{g_l}} \dx
        \,=\, k^2  \int_{P_m} |\uig|^2 \dx
        \,\geq\, 0
    \end{equation*}
    with $g := \sum_{j=1}^N v_jg_j$.
    This shows that $\bfF_N'[0]\chi_{P_m} \in \C^{N\times N}$ is self-adjoint and positive semidefinite.

    To see that the matrix is in fact positive definite, we recall from \cite[Thm.~3.27]{ColKre19} that $\uig$ can only vanish on the open subset $P_m\subs\Rd$ when $g=0$.
    However, the latter is equivalent to $v=0$.
\end{proof}

Now we are ready to discuss the solvability of \eqref{eq:MinimizationProblem}.
\begin{theorem}
\label{thm:MinimizerExactData}
    The minimization problem~\eqref{eq:MinimizationProblem} has the unique minimizer
    \begin{equation}
    \label{eq:MinimizerExactData}
        \hhat
        \,=\, \sum_{m=1}^M \min\{\qmin,\beta_m\} \chi_{P_m} \,.
    \end{equation}
    In particular, we have for $m=1,\ldots,M$ that $P_m\subs \supp(\hhat)$ if and only if $P_m\subs D$.
\end{theorem}

\begin{proof}
    Let $h = \sum_{m=1}^M a_m \chi_{P_m}$ be an arbitrary element of the admissible set $\Acal$ and let $\hhat$ be defined by~\eqref{eq:MinimizerExactData}.
    Then, the definition of $\Acal$ in~\eqref{eq:MinimizationProblem} immediately implies that~$h \leq \hhat$.
    Suppose that $\hhat-h\neq0$.
    Then
    \begin{equation*}
        \tau
        \,:=\, \hhat - h
        \,=\, \sum_{m=1}^M \tau_m \chi_{P_m}
        \qquad \text{with } \tau_m := \min\{\qmin,\beta_m\} - a_m \geq 0 \,,
    \end{equation*}
    and $\tau_m>0$ for at least one $m\in\{1,\ldots,M\}$.
    Let
    \begin{equation*}
        \lambda_1(\bfR(h))
        \,\geq\, \lambda_2(\bfR(h))
        \,\geq\, \cdots \,\geq\, \lambda_N(\bfR(h))
    \end{equation*}
    be the $N$ eigenvalues of $\bfR(h)$, and denote by
    \begin{equation*}
        \lambda_1(\bfR(\hhat))
        \,\geq\, \lambda_2(\bfR(\hhat))
        \,\geq\, \cdots \,\geq\, \lambda_N(\bfR(\hhat))
    \end{equation*}
    the $N$ eigenvalues of $\bfR(\hhat)$.
    Since
    \begin{equation*}
        \bfR(h)
        \,=\, \bfR(\hhat) + \sum_{m=1}^M \tau_m \bfF_N'[0]\chi_{P_m} \,,
    \end{equation*}
    the positive definiteness of the self-adjoint matrices $\bfF_N'[0]\chi_{P_m}$ for all $m\in\{1,\ldots,M\}$ established in Lemma~\ref{lmm:bfS_mPosDef} together with Weyl's inequality (see, e.g., \cite[p.~62]{Bha97}) show that
    \begin{equation*}
        \lambda_j(\bfR(h))
        \,>\, \lambda_j(\bfR(\hhat))
        \qquad \text{for } j=1,\ldots, N \,.
    \end{equation*}
    Accordingly,
    \begin{equation*}
        \sum_{\{ j \;|\; \lambda_j > 0\}} \lambda_j(\bfR(h))
        \,>\, \sum_{\{ j \;|\; \lambda_j > 0\}} \lambda_j(\bfR(\hhat)) \,,
    \end{equation*}
    which confirms that $\hhat$ is the unique minimizer of the minimization problem \eqref{eq:MinimizationProblem}.

    Finally, it is clear from \eqref{eq:MinimizerExactData} that $P_m\subs\supp\hhat$ if an only if $\beta_m>0$.
    Lemma~\ref{lmm:Pmbetam} says that this is the case if and only if $P_m\subs D$.
\end{proof}

Theorem~\ref{thm:MinimizerExactData} says that the unique minimizer of \eqref{eq:MinimizationProblem} coincides with the largest element in the admissible set $\Acal$.
Accordingly, it contains the same information about the support of the scatterer~$D$ as the linear constraints $\beta_m$, $m=1,\ldots,M$, from \eqref{eq:Defbetam}.
Lemma~\ref{lmm:Pmbetam} confirms that this information determines the support of the scatterer $D$ up with pixel accuracy.
However, the evaluation of~$\min(\qmin,\beta_m)$, $m=1,\ldots,M$, is relatively sensitive to noise in the observed far field data, and the corresponding reconstructions deteriorate with increasing noise level.
In the next section we couple the monotonicity-based approach \eqref{eq:MinimizationProblem} with a simple one-step linearization, and we show that when the coupling parameter is chosen appropriately, depending on the noise level, then the solution of this regularized problem converges to the solution of \eqref{eq:MinimizationProblem} as the noise level tends to zero.
This means, we use one-step linearization to stabilize the monotonicity-based regularization in the presence of noise.

\subsection{Regularization and stability analysis for noisy data}
\label{subsec:Stability}

Suppose $F^\delta[q] \in \Lcal(\LtSd)$ is a noisy version of the far field operator $F[q]$ satisfying
\begin{equation}
\label{eq:EstNoiseLevel}
    \|F^\delta[q] - F[q] \|_{\Lcal(\LtSd)}
    \,\leq\, \delta
\end{equation}
with some noise level $\delta > 0$.
Accordingly, let
\begin{equation}
\label{eq:FarfieldObservations_bfFNdelta}
    \bfF^\delta_N[q]
    \,:=\, \Bigl(
    \int_\Sd g_j \, \ol{F^\delta[q] g_l} \ds
    \Bigr)_{j,l=1}^N
    \,\in\, \C^{N\times N} \,,
\end{equation}
where $\{g_1,g_2,\ldots,g_N\}\subs\LtSd$ is the same orthonormal system as at the beginning of Section~\ref{sec:MonotonicityRegularization}.
Given such noisy far field observations we define the linearized residual ${\bfR^\delta: \Ycal \to \C^{N\times N}}$ by
\begin{equation}
    \label{eq:DefRdelta}
    \bfR^\delta (h)
    \,:=\, \real \bigl( \bfF_N^\delta[q] - \bfF_N'[0]h \bigr) \,.
\end{equation}
Let $\lambda_1(\bfR^\delta(h)) \geq \lambda_2(\bfR^\delta(h)) \geq \cdots \geq \lambda_N(\bfR^\delta(h))$ denote the eigenvalues of this self-adjoint matrix in descending order.
Then we replace the constrained minimization problem \eqref{eq:MinimizationProblem} by
\begin{subequations}
\label{eq:MinimizationProblemNoisy}
\begin{equation}
    \min_{h\in\Acal^\delta} \Bigl(
    \sum_{\{ j \;|\; \lambda_j > 0\}} \lambda_j(\bfR^\delta(h))
    + \alpha(\delta) \|\bfR^\delta(h)\|_F  \Bigr) \,,
\end{equation}
where $\alpha(\delta)>0$ is a regularization parameter, and the admissible set $\Acal^\delta$ is given by
\begin{equation}
    \label{eq:DefAdelta}
    \Acal^\delta
    \,:=\, \Bigl\{
    h = \sum_{m=1}^M a_m \chi_{P_m} \;\Big|\;
    0 \leq a_m \leq \min(\qmin,\beta_m^\delta)
    \Bigr\} \,.
\end{equation}
\end{subequations}
Here we use
\begin{equation}
\label{eq:Defbetamdelta}
    \beta_m^\delta
    \,:=\,
    \max \bigl\{ \beta\geq 0 \;\big|\;
    \real ( F^\delta[q] - \beta F'[0]\chi_{P_m} ) \geq_{d_m} -\delta I \bigr\} \,,
\end{equation}
instead of the linear constraint $\beta_m$ from \eqref{eq:Defbetam}, where as before $d_m = d(\qmin,P_m)$ is the number of negative eigenvalues of $\real(F[q]-\qmin F'[0]\chi_{P_m})$ in case this number is finite;
if $d(\qmin,P_m)$ is not finite, we set $\beta_m^\delta=0$.

\begin{lemma}
\label{lmm:betamdelta}
    For all $m=1,\ldots, M$ we have that $\beta_m\leq\beta_m^\delta$.
\end{lemma}

\begin{proof}
    Let $g\in\LtSd$.
    Using \eqref{eq:EstNoiseLevel} we find that
    \begin{equation*}
    \begin{split}
        \big| \big\langle g , \real(F^\delta[q] - F[q]) g \big\rangle_{\LtSd} \big|
        &\,\leq\, \|g\|_{\LtSd} \| \real(F^\delta[q] - F[q]) g \|_{\LtSd} \\
        &\,\leq\, \|g\|_{\LtSd}^2 \| F^\delta[q] - F[q] \|_{\Lcal(\LtSd)}
        \,\leq\, \delta \|g\|_{\LtSd}^2 \,.
    \end{split}
    \end{equation*}
Thus, $-\delta I \leq \real(F^\delta[q] - F[q]) \leq \delta I$ in quadratic sense, and for $m=1,\dots,M$ the definition of $\beta_m$ in \eqref{eq:Defbetam} and \eqref{eq:EstNoiseLevel} show that
\begin{equation}
\label{eq:ProofLmmbetamdelta}
    \real \bigl( F^\delta[q] - \beta F'[0]\chi_{P_m} \bigr)
    \,=\, \real \bigl( F[q] - \beta F'[0]\chi_{P_m} \bigr)
    + \real \bigl( F^\delta[q] - F[q] \bigr)
    \,\geq_{d_m}\, -\delta I
\end{equation}
for all $\beta\in[0,\beta_m]$.
Accordingly, \eqref{eq:Defbetamdelta} yields $\beta_m\leq \beta_m^\delta$.
\end{proof}

\begin{remark}
(a)\; The linear constraint $\beta_m^\delta$ still contains some information about the support of the scatterer.
Combining Lemma~\ref{lmm:betamdelta} with Lemma~\ref{lmm:Pmbetam} we obtain that $P_m\subs D$ implies $\beta_m^\delta \geq \qmin$, i.e.,
$P_m\not\subs D$ if $\beta_m^\delta < \qmin$.

(b)\; Using the same arguments as in the proof of Theorem~\ref{thm:MinimizerExactData} it follows that for any fixed $\delta>0$ the constrained optimization problem
\begin{equation}
    \label{eq:MinimizationProblemNoisyNoReg}
    \min_{h\in\Acal^\delta} \Bigl(
    \sum_{\{ j \;|\; \lambda_j > 0\}} \lambda_j(\bfR^\delta(h)) \Bigr) \,,
\end{equation}
    with $\Acal^\delta$ as defined in \eqref{eq:DefAdelta} has the unique minimizer
    \begin{equation*}
        \htilde^\delta
        \,=\, \sum_{m=1}^M \min\{\qmin,\beta_m^\delta\} \chi_{P_m} \,.
    \end{equation*}
    This means that for the degenerate case when $\alpha(\delta)=0$ the minimization problem \eqref{eq:MinimizationProblemNoisy} just recovers the largest element in the admissible set $\Acal^\delta$.

(c)\; Minimizing the Frobenius norm of the linearized residual $\bfR^\delta(h)$ as in \eqref{eq:MinimizationProblemNoisy} makes particular sense for small wave numbers $k$ (i.e., at low frequencies) or for small contrast functions $q$, when the Born approximation $\bfF_N'[0]q$ constitutes a good approximation of $\bfF_N[q]$ (see, e.g., \cite{Kir17}).
In this case we expect the regularized functional in \eqref{eq:MinimizationProblemNoisy} to yield significantly better reconstructions of the support of the scatterer $D$ when compared to minimizers of the unregularized functional~\eqref{eq:MinimizationProblemNoisyNoReg}.~\hfill$\lozenge$
\end{remark}

In the next theorem we establish that minimizers of \eqref{eq:MinimizationProblemNoisy} converge to the minimizer of \eqref{eq:MinimizationProblem} as the noise level $\delta$ tends to zero, provided the coupling parameter $\alpha = \alpha(\delta)$ is chosen appropriately with respect to the noise level.
This means that the regularized solutions converge to a piecewise constant function in $\Ycal$ that determines the support of the scatterer $D$ up to pixel partition as the noise level tends to zero.

\begin{theorem}
\label{thm:Stability}
Suppose that $\alpha:[0,\infty)\to [0,\infty)$ is an a priori parameter choice rule satisfying $\lim_{\delta\to 0} \alpha(\delta) = 0$.
Then, for any $\delta>0$, the minimization problem \eqref{eq:MinimizationProblemNoisy} attains a minimizer.
Furthermore, denoting by $\hhat=\sum_{m=1}^M \min\{\qmin,\beta_m\} \chi_{P_m}$ the unique minimizer of \eqref{eq:MinimizationProblem}, and by
$\hhat^\delta = \sum_{m=1}^M \ahat_m^\delta \chi_{P_m}$ any minimizer of \eqref{eq:MinimizationProblemNoisy}, we have that $\hhat^\delta \to \hhat$ in $\Ycal$ as $\delta\to 0$.
\end{theorem}

\begin{proof}
    The existence of at least one minimizer of \eqref{eq:MinimizationProblemNoisy} follows from the continuity of the cost functional
    \begin{equation*}
        \Ycal \ni h \mapsto \sum_{\{ j \;|\; \lambda_j > 0\}} \lambda_j(\bfR^\delta(h)) + \alpha(\delta) \|\bfR^\delta(h)\|_F
    \end{equation*}
    and the compactness of the admissible set $\Acal^\delta$.

    Now let $(\delta_n)_{n\in\N} \subs (0,\infty)$ be a sequence converging to zero, and denote by $\hhat^{\delta_n} = \sum_{m=1}^M \ahat_m^{\delta_n} \chi_{P_m}$ an associated sequence of minimizers.
    Since $0< \ahat_m^{\delta_n}< \qmin$ for all $m\in\{1,\ldots,M\}$ and $n\in\N$, the sequence $(\ahat^{\delta_n}_1,\ldots,\ahat^{\delta_n}_M)_{n\in\N} \subs \R^M$ is bounded, and thus has a convergent subsequence.
    With an abuse of notation, we also denote this subsequence by $(\ahat^{\delta_n}_1,\ldots,\ahat^{\delta_n}_M)_{n\in\N}$.
    Moreover, we denote its limit by $(a_1,\ldots,a_M)$, and observe that $0\leq a_m\leq \qmin$ for all $m\in\{1,\ldots,M\}$.

    From \eqref{eq:EstNoiseLevel} we see that $F^{\delta_n}[q]\to F[q]$ in $\Lcal(\LtSd)$ with respect to the operator norm as~${n\to\infty}$.
    Thus, for any $m\in\{1,\ldots,M\}$,
    \begin{equation*}
    F^{\delta_n}[q] - \ahat_m^{\delta_n} F'[0]\chi_{P_m}
    \,\to\, F[q] - a_m F'[0]\chi_{P_m}
    \qquad \text{as } n\to\infty
    \end{equation*}
    in $\Lcal(\LtSd)$, and using \eqref{eq:ProofLmmbetamdelta} we find that
    \begin{equation*}
        \real \bigl( F[q] - a_m F'[0]\chi_{P_m} \bigr)
        \,=\, \lim_{n\to\infty} \real \bigl( F^{\delta_n}[q] - \ahat_m^{\delta_n} F'[0]\chi_{P_m} \bigr)
        \,\geq_{d_m} \lim_{n\to\infty} (-\delta_n I)
        \,=\, 0
    \end{equation*}
    in quadratic sense.
    Therefore, recalling \eqref{eq:Defbetam} we obtain that $a_m\leq\beta_m$ for all $m\in\{1,\ldots,M\}$.
    In particular, we have shown that $h:=\sum_{m=1}^M a_m \chi_{P_m}$ belongs to the admissible set $\Acal$ of \eqref{eq:MinimizationProblem}.

    Since $\min\{\qmin,\beta_m\} \leq \min\{\qmin,\beta_m^{\delta_n}\}$ for all $m\in\{1,\ldots,M\}$ and $n\in\N$ by Lemma~\ref{lmm:betamdelta}, we find that $\hhat \in \Acal^{\delta_n}$ for all $n\in\N$.
    Accordingly,
    \begin{equation}
    \label{eq:ProofThmStability1}
        \sum_{\{ j \;|\; \lambda_j > 0\}} \lambda_j(\bfR^{\delta_n}(h^{\delta_n}))
        + \alpha(\delta_n) \|\bfR^{\delta_n}(h^{\delta_n})\|_F
        \,\leq\, \sum_{\{ j \;|\; \lambda_j > 0\}} \lambda_j(\bfR^{\delta_n}(\hhat))
        + \alpha(\delta_n) \|\bfR^{\delta_n}(\hhat)\|_F
    \end{equation}
    for all $n\in\N$ by the optimality property of $h^{\delta_n}$.
    Moreover, it follows immediately from our assumptions on $\alpha=\alpha(\delta_n)$ that
    \begin{equation*}
        \sum_{\{ j \;|\; \lambda_j > 0\}} \lambda_j(\bfR^{\delta_n}(h^{\delta_n}))
        + \alpha(\delta_n) \|\bfR^{\delta_n}(h^{\delta_n})\|_F
        \,\to\, \sum_{\{ j \;|\; \lambda_j > 0\}} \lambda_j(\bfR(h))
        \qquad \text{as } n\to\infty \,,
    \end{equation*}
    and
    \begin{equation*}
        \sum_{\{ j \;|\; \lambda_j > 0\}} \lambda_j(\bfR^{\delta_n}(\hhat))
        + \alpha(\delta_n) \|\bfR^{\delta_n}(\hhat)\|_F
        \,\to\, \sum_{\{ j \;|\; \lambda_j > 0\}} \lambda_j(\bfR(\hhat))
        \qquad \text{as } n\to\infty \,.
    \end{equation*}
    Therefore, \eqref{eq:ProofThmStability1} yields
    \begin{equation*}
        \sum_{\{ j \;|\; \lambda_j > 0\}} \lambda_j(\bfR(h))
        \,\leq\, \sum_{\{ j \;|\; \lambda_j > 0\}} \lambda_j(\bfR(\hhat)) \,.
    \end{equation*}
    Recalling that $\hhat$ is the unique minimizer of \eqref{eq:MinimizationProblem} and that $h\in\Acal$, we obtain that $h=\hhat$.

    Since this is true for any convergent subsequence of the original sequence $(\ahat^{\delta_n}_1,\ldots,\ahat^{\delta_n}_M)_{n\in\N}$, we have shown that $\hhat^{\delta_n} \to h$ as $n\to\infty$ also for the original sequence $(\delta_n)_{n\in\N}$, which completes the proof.
\end{proof}

\section{Numerical examples}
\label{sec:Numerics}
We illustrate our theoretical findings by some numerical examples in $\Rtwo$.
To simulate a finite number of far field observations, which are then used as input data for the reconstruction algorithm,  we numerically approximate far field data $\uinfty(\xhat_l;\theta_m)$ corresponding to plane wave incident fields
\begin{equation*}
\ui(x;\theta_m)
\,:=\, \rme^{\rmi k x\cdot\theta_m} \,, \qquad x\in\Rtwo \,,
\end{equation*}
for $N$ equidistant observation and incident directions
\begin{equation*}
    \xhat_l,\theta_m
    \in \{ (\cos\phi_n,\sin\phi_n)\in\Sone \;|\;
    \phi_n = (n-1)2\pi/N \,,\; n=1,\ldots,N \} \,,
\end{equation*}
$1\leq l,m\leq N$.
If the support $D$ of the scatterer is contained in a ball $\BR$ of radius $R>0$ around the origin, then choosing $N\gtrsim 2kR$ is sufficient to fully resolve the relevant information contained in these far field patterns (see, e.g., \cite{GriSyl17}).
Accordingly, the matrix
\begin{equation}
\label{eq:DefF_N(q)}
    \bfF_N[q]
    \,=\, \frac{2\pi}{N} \bigl[ \uinfty(\xhat_l;\theta_m) \bigr]
    \in \C^{N\times N}
\end{equation}
is unitary equivalent to the matrix $\bfF_N[q]$ from \eqref{eq:DefFNq} for the orthonormal system $\{g_1,g_2,\ldots,g_N\}\subs L^2(S^1)$ with $g_n = \rme^{\rmi (n-N/2)\arg(\cdot)} / \sqrt{2\pi}$ for $n=1,\ldots,N$.

To simulate noisy far field observations
$\bfF_N^\delta[q]$ as in \eqref{eq:FarfieldObservations_bfFNdelta} satisfying
\eqref{eq:EstNoiseLevel} for some noise level $\delta>0$, we generate a complex random matrix $\bfE\in\C^{N\times N}$ with uniformly distributed real and imaginary parts and evaluate
\begin{equation}
    \label{eq:DefF_N^delta(q)}
    \bfF_N^\delta[q]
    \,=\, \bfF_N[q] + \delta \frac{\bfE}{\|\bfE\|_2} \,.
\end{equation}

We consider an equidistant grid on the quadratic region of interest
\begin{equation}
\label{eq:DiscretizationOmega}
    \Omega
    \,=\, \Bigl[ -\frac{R}{\sqrt{2}} , \frac{R}{\sqrt{2}} \Bigr]^2
    \,=\, \bigcup_{m=1}^M P_m \,,
\end{equation}
with quadratic pixels $P_m = z_m + [-\frac{\ell}{2},\frac{\ell}{2}]^2$\,, $1\leq m\leq M$, where $z_m\in\Omega$ denotes the center of~$P_m$ and~$\ell>0$ its side length.
A short calculation shows that for each pixel $P_m$ the matrix ${\bfS_m := \bfF'_N[0]\chi_{P_m}}$ from \eqref{eq:DefFN'0} is given by
\begin{equation}
\label{eq:DefSm}
    \bfS_m
    \,=\, \frac{2\pi}{N}
    \Bigl[ (k\ell)^2 \rme^{\rmi k z_m\cdot(\theta_m-\theta_l)} \sinc\Bigl(\frac{k\ell}{2}(\theta_m-\theta_l)_1\Bigr) \sinc\Bigl(\frac{k\ell}{2}(\theta_m-\theta_l)_2\Bigr)
    \Bigr]_{1\leq l,m\leq N} \in \C^{N\times N} \,.
\end{equation}
Accordingly, we can rewrite the linearized residual from \eqref{eq:DefRdelta} as
\begin{equation*}
    \bfR^\delta(h)
    \,=\, \bfV^\delta - \sum_{m=1}^M a_m \bfS_m \,,
\end{equation*}
where we use the shortcut $\bfV^\delta := \real\bigl( \bfF_N^\delta[q] \bigr)$ and as before $h=\sum_{m=1}^M a_m\chi_{P_m}$.

Next we detail our approach to determine the linear constraint $\beta_m^\delta$ from~\eqref{eq:Defbetamdelta}, which satisfies
\begin{equation}
\label{eq:Approximation_betamdelta}
    \beta_m^\delta
    \,\approx\, \max \bigl\{ \beta\geq 0 \;\big|\;
    \bfV^\delta - \beta \bfS_m \geq_{d_m} -\delta I \bigr\} \,.
\end{equation}
We approximate the number $d_m = d(\qmin,P_m)$ of negative eigenvalues of ${\real(F[q]-\qmin F'[0]\chi_{P_m})}$ by the number $\dtilde_m$ of negative eigenvalues of
\begin{equation*}
    \bfV^\delta - \qmin\bfS_m + \delta \bfI \,\in\, \C^{N\times N} \,,
\end{equation*}
where $\bfI\in\C^{N\times N}$ denotes the identity matrix.
This is done by evaluating the eigenvalues of this self-adjoint matrix for each $m\in\{1,\ldots,M\}$ numerically.

Then, we consider the generalized eigenvalue problem for the matrix pencil
\begin{equation*}
    \bfS_m-\lambda(\bfV^\delta+\delta \bfI) \,, \qquad \lambda\in\C \,.
\end{equation*}
Since $\bfS_m$ is invertible by Lemma~\ref{lmm:bfS_mPosDef},
\begin{equation*}
    \det(\bfS_m-\lambda(\bfV^\delta+\delta\bfI)) \,=\, 0
    \qquad \text{if and only if} \qquad
    \det(I-\lambda\bfS_m^{-1}(\bfV^\delta+\delta\bfI)) \,=\, 0 \,.
\end{equation*}
Accordingly, recalling that both $\bfS_m$ and $\bfV^\delta+\delta\bfI$ are self-adjoint, the generalized eigenvalues of $\bfS_m-\lambda(\bfV^\delta+\delta\bfI)$ are real (including possibly $\pm\infty$), and there exists a corresponding orthonormal basis of $\C^N$ consisting of eigenvectors of $\bfS_m-\lambda\bfV^\delta$.
More precisely, $\lambda\in(\R\setminus\{0\})\cup\{\pm\infty\}$ is a generalized eigenvalue of $\bfS_m-\lambda(\bfV^\delta+\delta\bfI)$ if and only if $1/\lambda\in\R$ is an eigenvalue of the self-adjoint matrix $\bfS_m^{-1}(\bfV^\delta+\delta\bfI)$.

We denote these generalized eigenvalues by $\lambda_1\geq\lambda_2\geq\cdots\geq\lambda_N$, and an associated orthonormal basis of eigenvectors by $\{x_1,\ldots,x_N\}$.
These can for instance be stably evaluated using the QZ-algorithm; see \cite[Sec.~7.7]{GolVan13}.
Then,
\begin{equation*}
    (\bfV^\delta +\delta\bfI)x_j
    \,=\, \frac{1}{\lambda_j} \bfS_m x_j \,, \qquad j=1,\ldots,N \,.
\end{equation*}
Choosing $\beta_m^* := 1/\lambda_1$, and denoting by $V_*$ the span of all eigenvectors corresponding to non-positive generalized eigenvalues, we obtain that
\begin{equation*}
    x^* (\bfV^\delta+\delta\bfI) x
    \,\geq\, \beta_m^* x^* \bfS_m x \qquad \text{for all } x\in V_*^\perp \,.
\end{equation*}
If $\dim(V_*)=d_*$, we use $\beta_m^*$ to approximate $\beta_m^\delta$.
On the other hand, if $r := d_*-\dim(V^*) > 0$, then we choose $\widetilde{\beta_m^*} := 1/\lambda_{r+1}$ instead.

As has been observed in \cite{EbeHarWan25}, the sum of all positive eigenvalues of the self-adjoint matrix $\bfR^\delta(h)$ in \eqref{eq:MinimizationProblemNoisy} can be written as a semidefinite program,
\begin{equation*}
    \sum_{\{j\;|\;\lambda_j>0\}} \lambda_j\bigl( \bfR^\delta(h) \bigr)
    \,=\, \min_{\substack{\bfX\geq 0\\ \bfX-\bfR^\delta(h)\geq 0}} \trace(\bfX) \,.
\end{equation*}
Hence, the minimization problem \eqref{eq:MinimizationProblemNoisy} is equivalent to
\begin{equation}
\label{eq:RegTraceMinimization}
    \min_{\substack{h\in\Acal^\delta\\\bfX\geq 0\\\bfX-\bfR^\delta(h)\geq 0}}
    \bigl( \trace(\bfX) + \alpha(\delta) \| \bfR^\delta(h) \|_F \bigr) \,.
\end{equation}
To solve this convex minimization problem (see, \cite[Rem.~6]{EbeHarWan25}) we used CVX, a package for specifying and solving convex programs \cite{CVX1,CVX2}.

\begin{example}
\label{exa:1}
In our numerical examples we consider two scatterers, a kite and an ellipse, with constant contrast function $q=1$ (kite) and $q=2$ (ellipse) as shown in Figures~\ref{fig:Exa_k0.5_delta0.01}--\ref{fig:Exa_k1.0_delta0.05} (dashed lines).
We simulate the far field matrix $\bfF_N[q]\in\C^{32\times 32}$ from \eqref{eq:DefF_N(q)} for $32$ incident and observation directions using a Nystr\"om method for a boundary integral formulation of the scattering problem for two wave numbers $k=0.5$ and $k=1$.
Therewith, we compute noisy data $\bfF_N^\delta[q]$ for $\delta=0.01$ and $\delta=0.05$ as described in \eqref{eq:DefF_N^delta(q)}, and accordingly $\bfV^\delta = (\bfF_N^\delta[q]+\bfF_N^\delta[q]^*)/2$.

We consider $\Omega = [-5,5]\times [-5,5]$ for the quadratic region of interest and divide it into~${M = 32^2}$ Pixels $P_m$, $m=1,\ldots,M$, of side length $\ell\approx0.31$ as described in~\eqref{eq:DiscretizationOmega}.
For each Pixel the self-adjoint matrix $\bfS_m$ can be evaluated using \eqref{eq:DefSm}.
In the numerical implementation of the monotonicity-based regularization~\eqref{eq:RegTraceMinimization} we use $\alpha(\delta) = \delta$ for the regularization parameter in~\eqref{eq:RegTraceMinimization}.
The linear constraints~$\beta_m^\delta$, $m=1,\ldots, M$, in the definition of the admissible set $\Acal_\delta$ from~\eqref{eq:MinimizationProblemNoisy} are approximated by~$\beta_m^*$, which are evaluated by solving generalized eigenvalue problems.
Furthermore, we put $\qmin=1$ in \eqref{eq:MinimizationProblemNoisy}.

\begin{figure}[th]
    \centering
    \includegraphics[height=3.7cm]{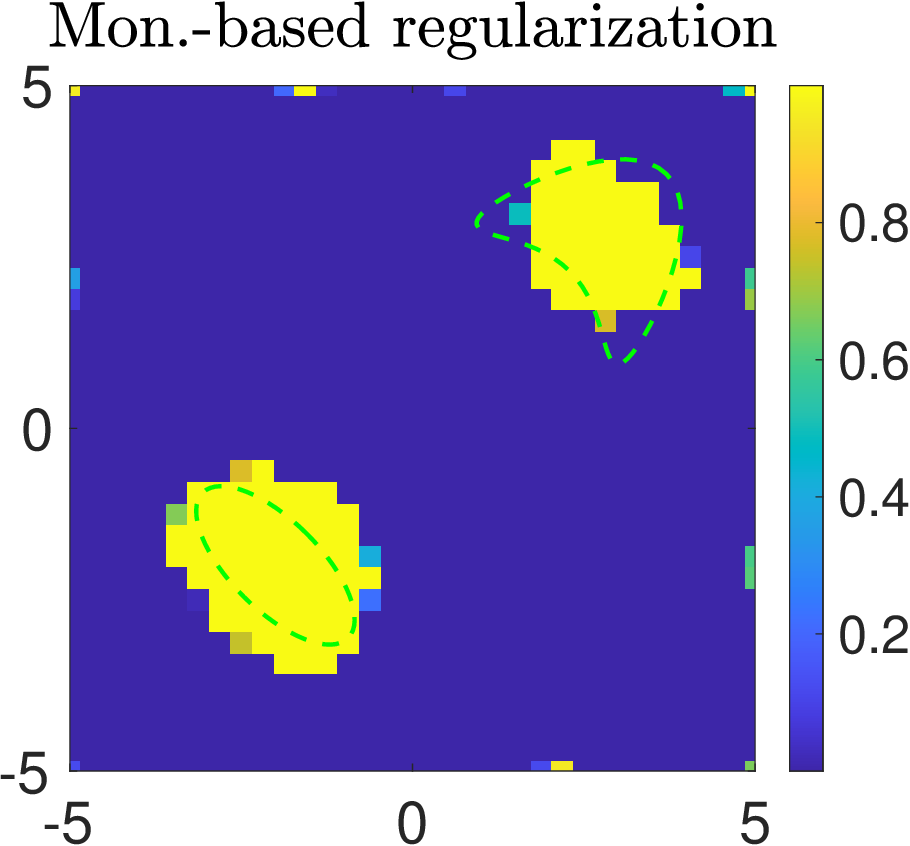}
    \includegraphics[height=3.7cm]{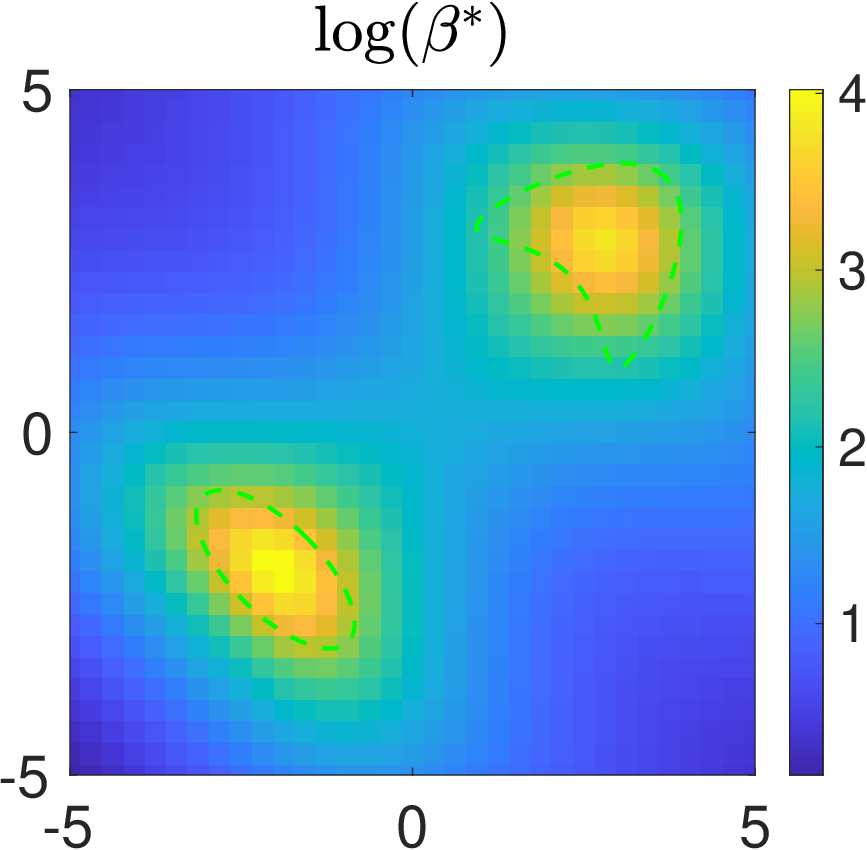}
    \includegraphics[height=3.7cm]{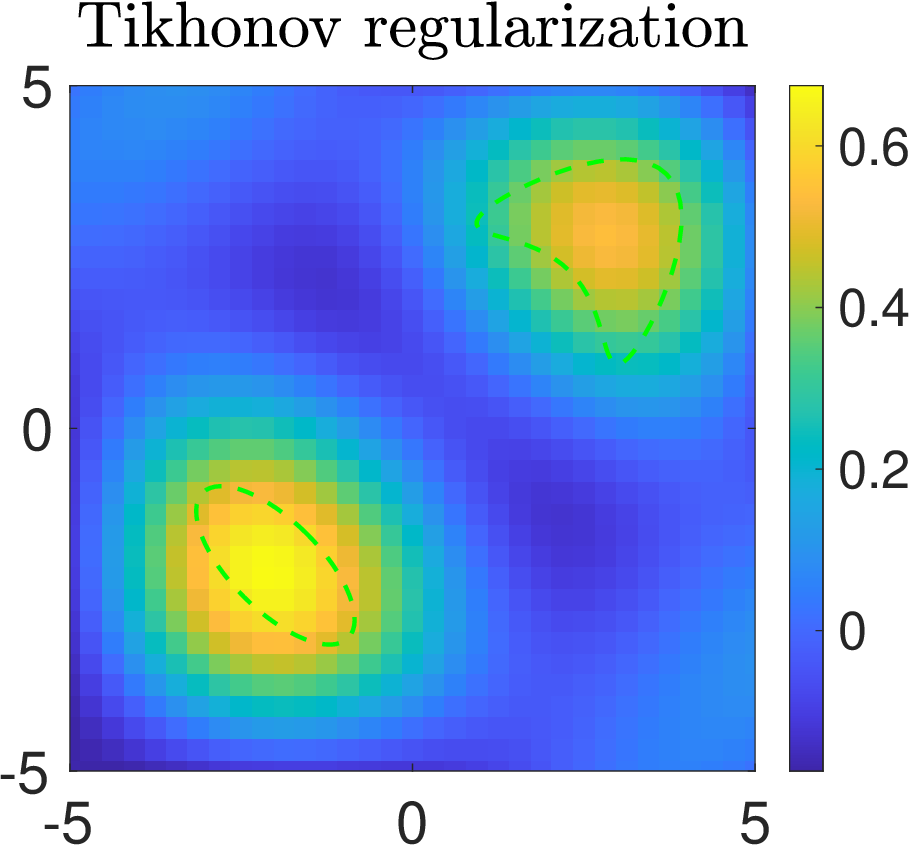}
    \includegraphics[height=3.7cm]{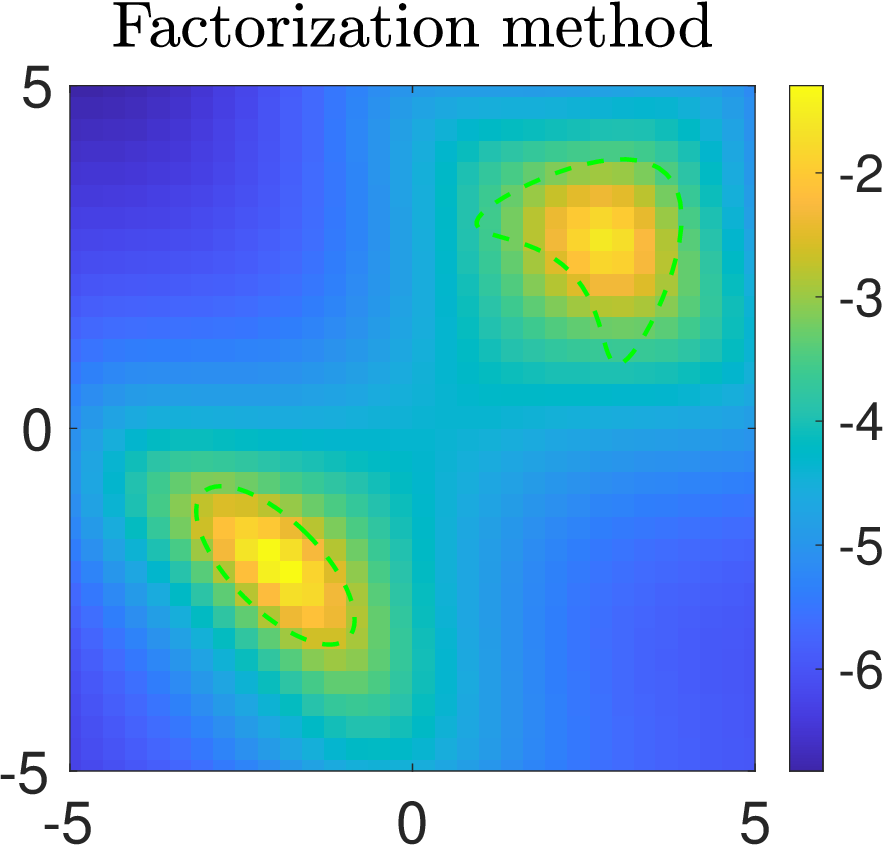}
    \caption{\small Reconstructions for wave number $k=0.5$ and noise level $\delta=0.01$.}
    \label{fig:Exa_k0.5_delta0.01}
\end{figure}

\begin{figure}[th]
    \centering
    \includegraphics[height=3.7cm]{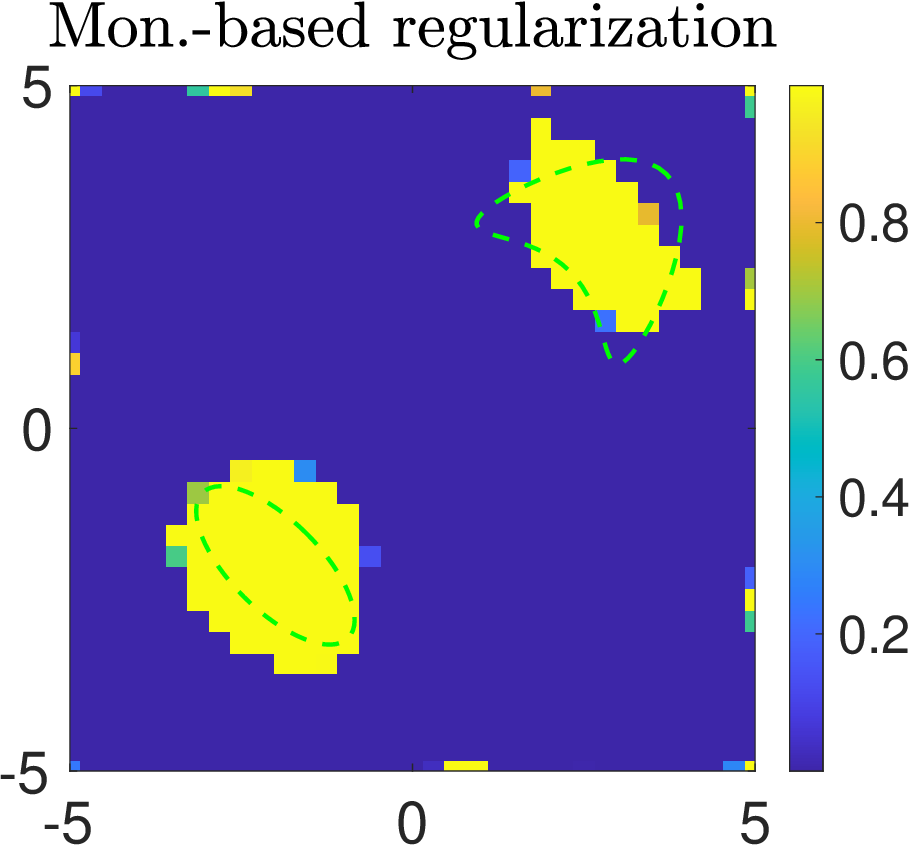}
    \includegraphics[height=3.7cm]{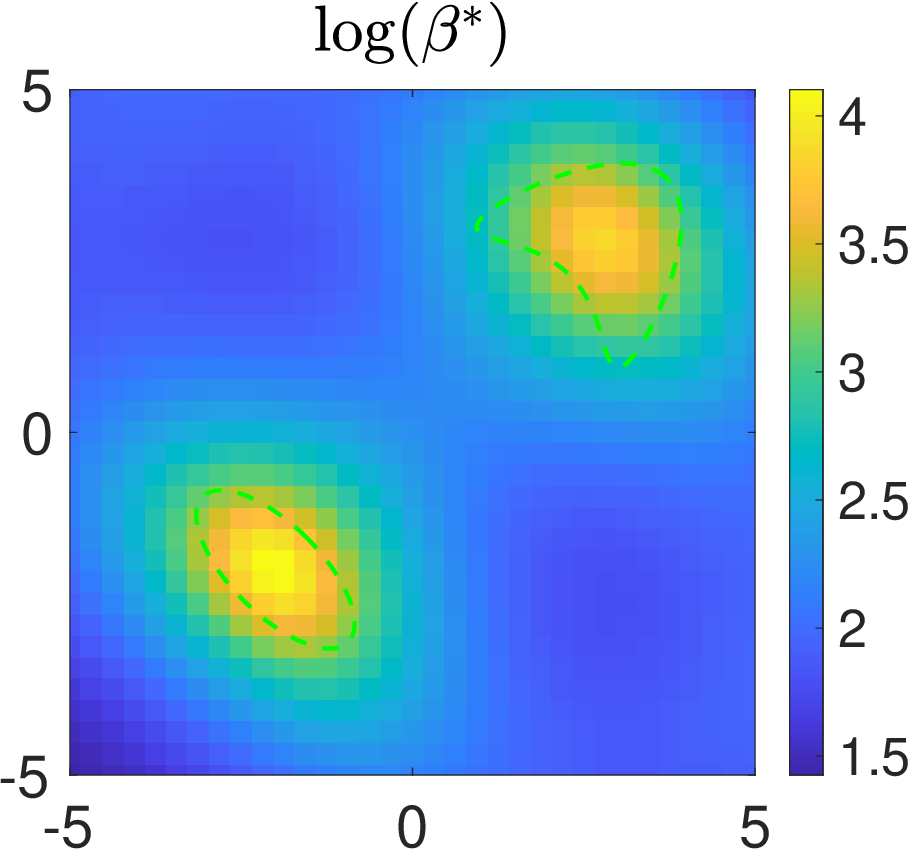}
    \includegraphics[height=3.7cm]{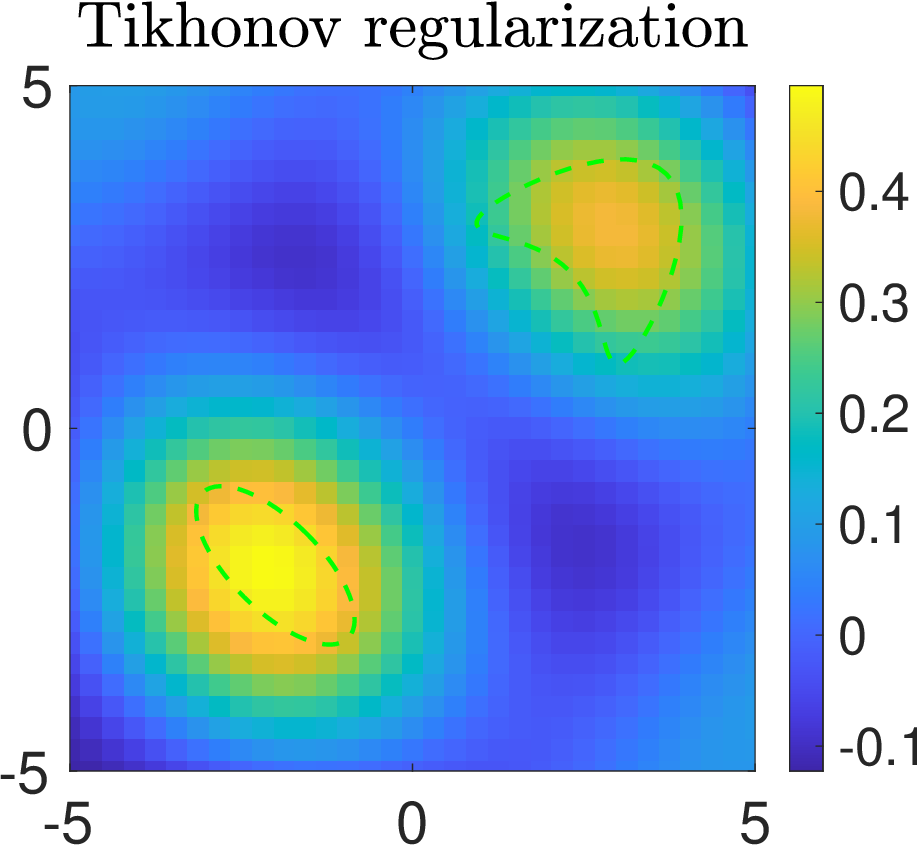}
    \includegraphics[height=3.7cm]{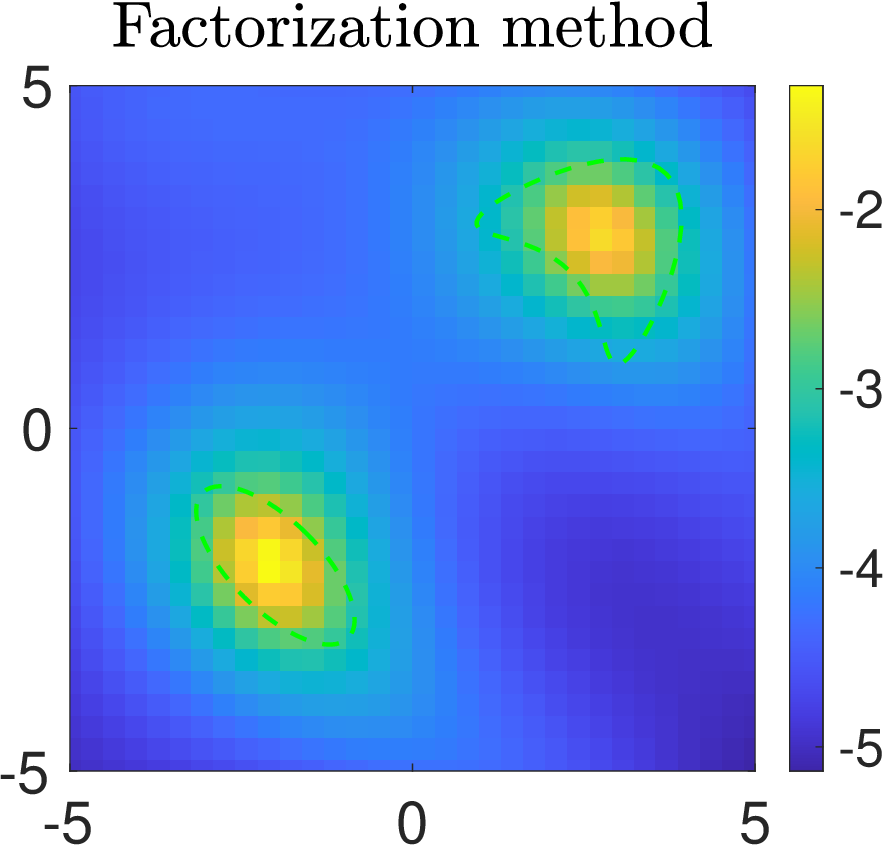}
    \caption{\small Reconstructions for wave number $k=0.5$ and noise level $\delta=0.05$.}
    \label{fig:Exa_k0.5_delta0.05}
 \end{figure}

\begin{figure}[th]
    \centering
    \includegraphics[height=3.7cm]{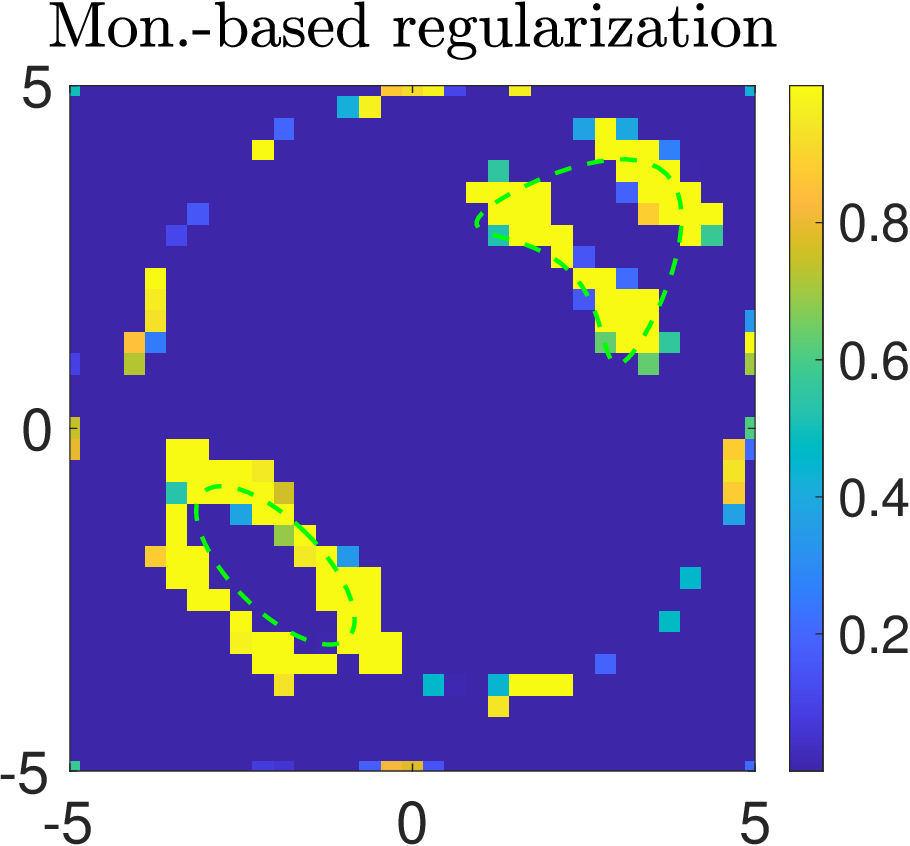}
    \includegraphics[height=3.7cm]{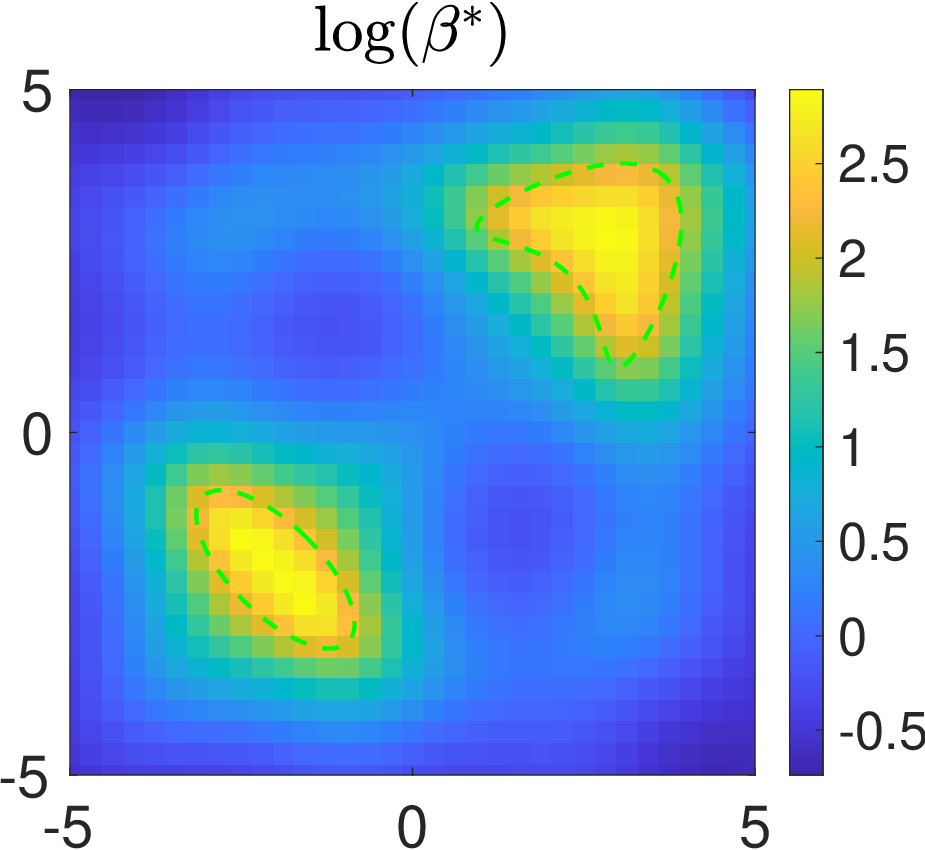}
    \includegraphics[height=3.7cm]{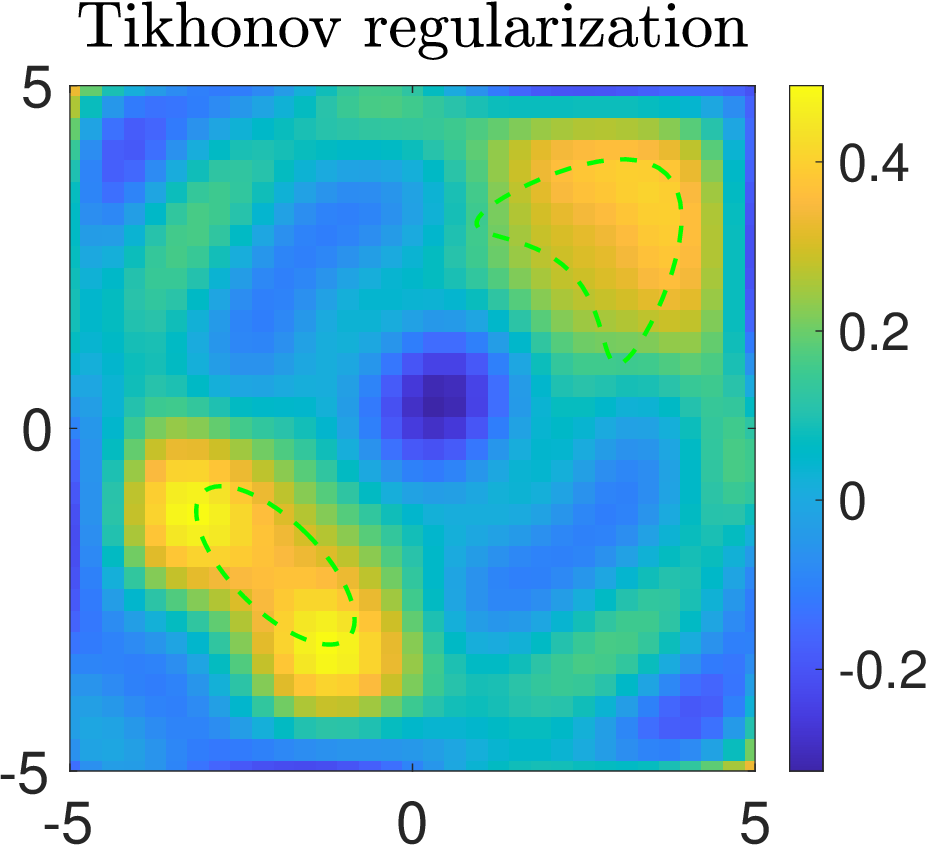}
    \includegraphics[height=3.7cm]{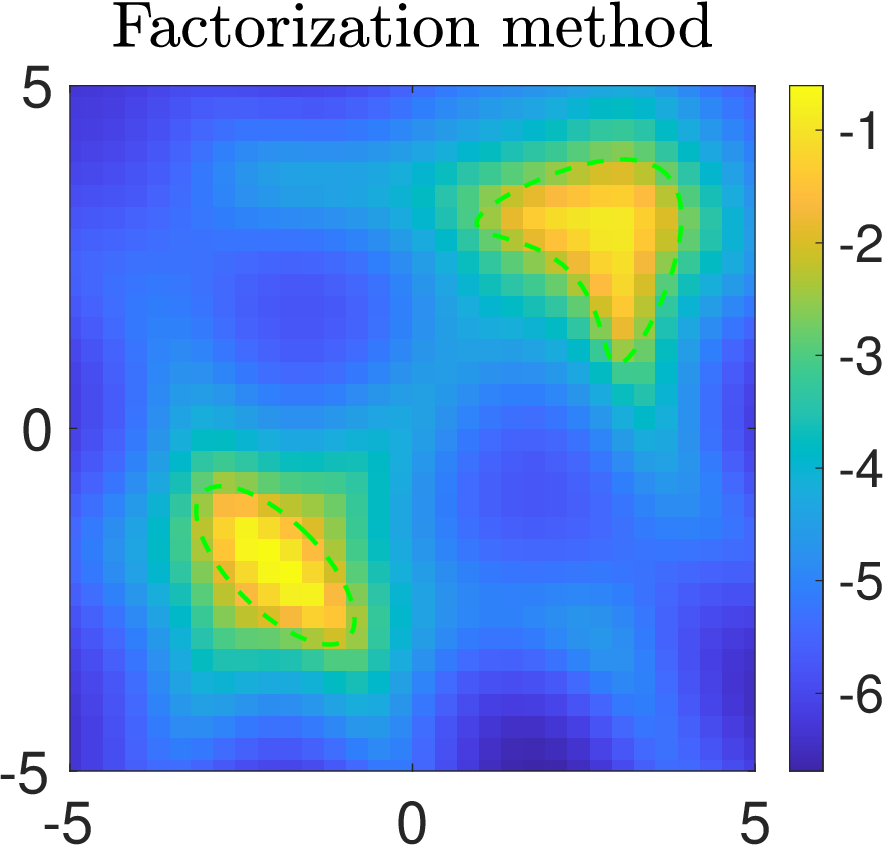}
    \caption{\small Reconstructions for wave number $k=0.5$ and noise level $\delta=0.01$.}
    \label{fig:Exa_k1.0_delta0.01}
\end{figure}

\begin{figure}[th]
    \centering
    \includegraphics[height=3.7cm]{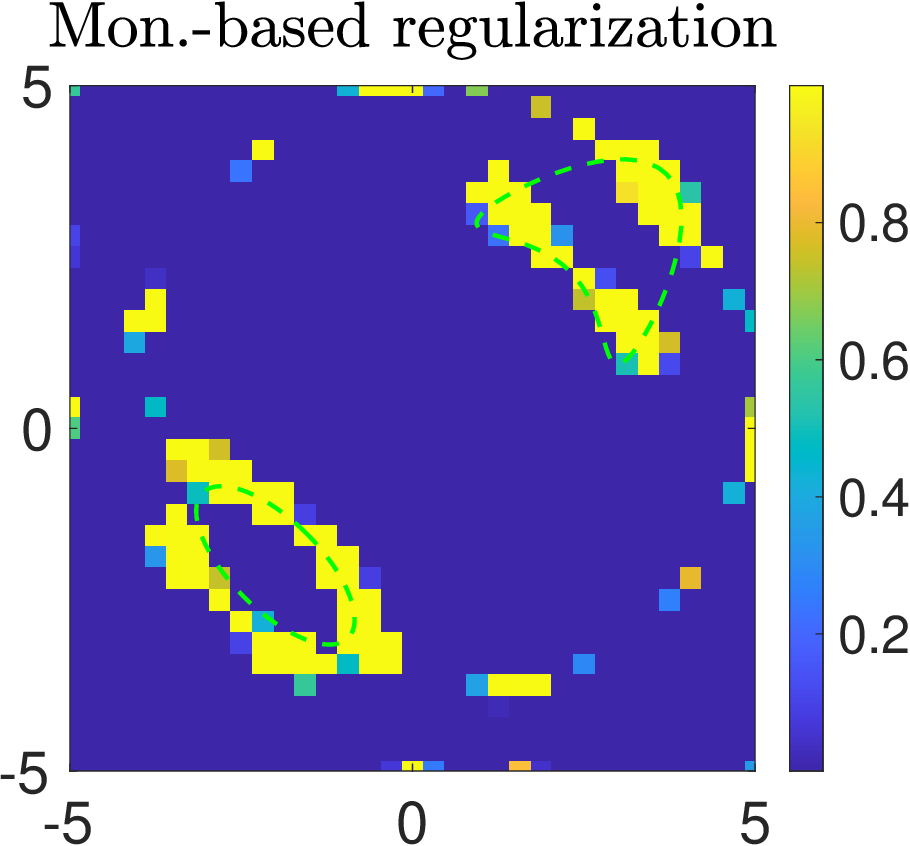}
    \includegraphics[height=3.7cm]{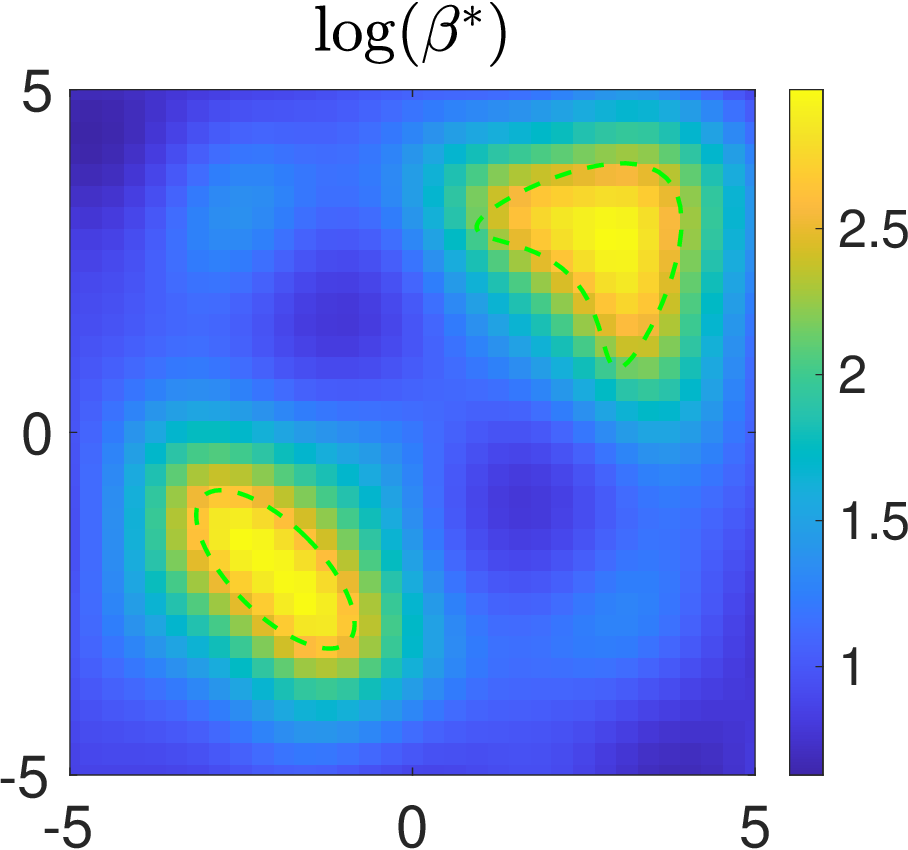}
    \includegraphics[height=3.7cm]{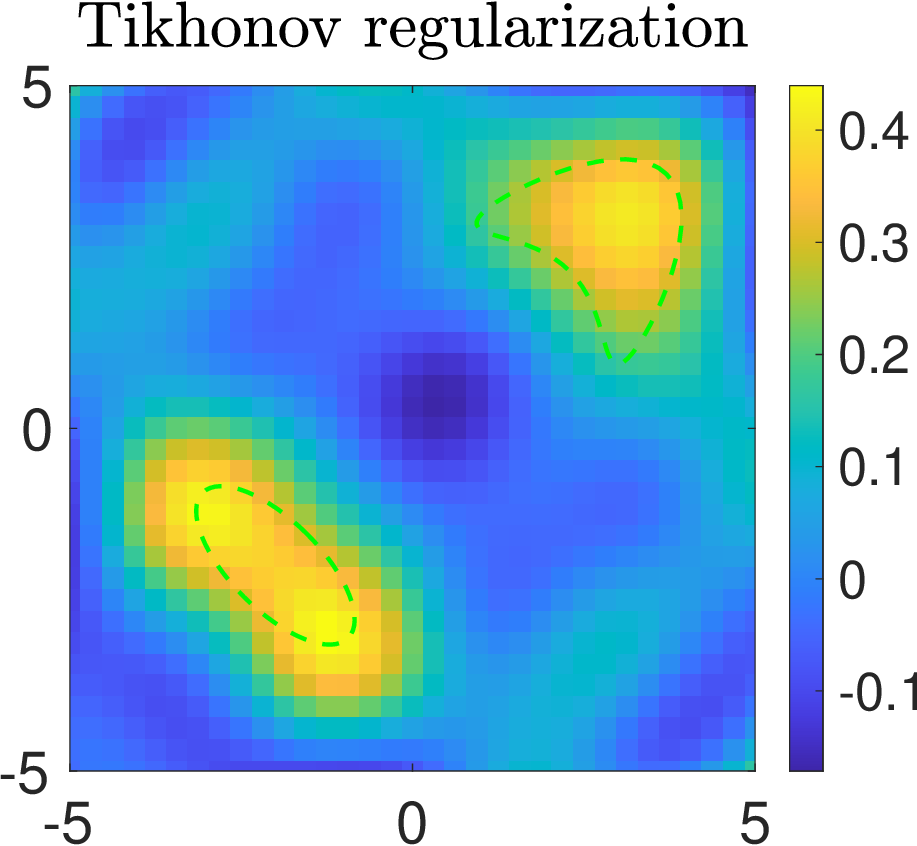}
    \includegraphics[height=3.7cm]{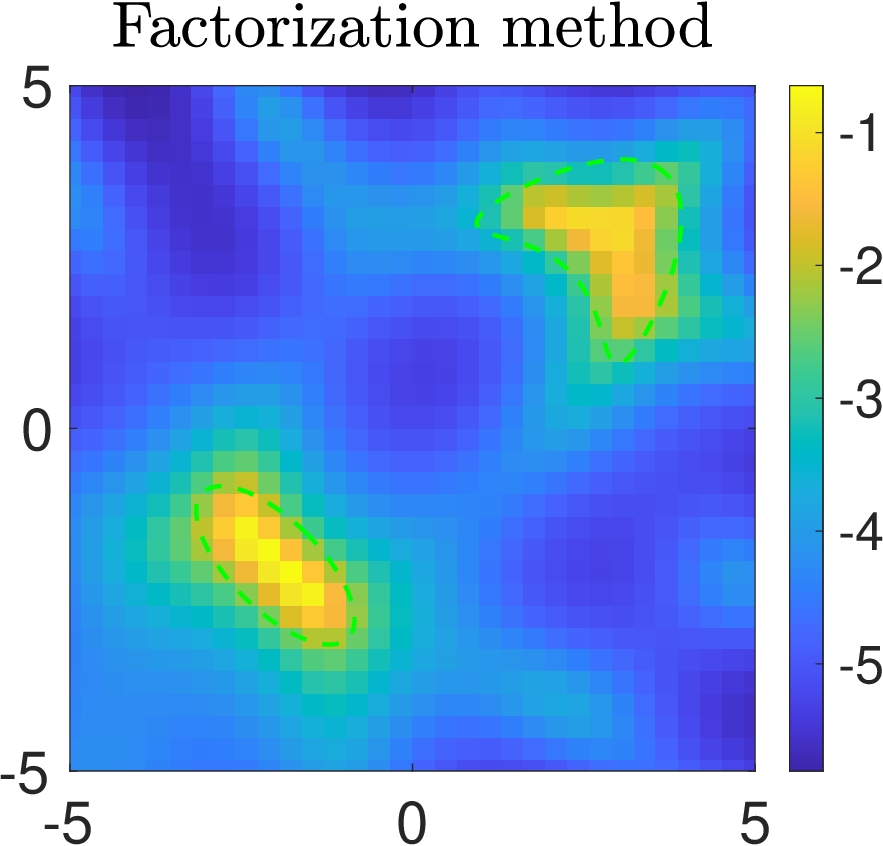}
    \caption{\small Reconstructions for wave number $k=0.5$ and noise level $\delta=0.05$.}
    \label{fig:Exa_k1.0_delta0.05}
\end{figure}

In Figures~\ref{fig:Exa_k0.5_delta0.01}--\ref{fig:Exa_k1.0_delta0.05} (right) we present numerical results for the two different wave numbers~(${k=0.5}$ and $k=1$) and for the two different noise levels~($\delta=0.01$ and $\delta=0.05$).
The first columns in these plots display the reconstructions obtained by the novel monotonicity-based regularization method.
In the second column we show plots of the linear constraints $\beta^*$ that are obtained from the monotonicity test \eqref{eq:Approximation_betamdelta} by solving generalized eigenvalue problems.
To assess the performance of the new method compared to a plain linearization approach, we include reconstructions obtained by a simple Tikhonov regularization of the linearized problem in the third column of these plots.
To this end, we minimize
\begin{equation*}
   \min_{h\in \Bcal} \bigl( \|\bfR^\delta(h)\|_F^2 + \delta \|h\|_{L^2(\Omega)}^2 \bigr) \,,
\end{equation*}
where the admissible set $\Bcal$ is given by
\begin{equation*}
    \Bcal
    \,:=\, \Bigl\{
    h = \sum_{m=1}^M a_m \chi_{P_m} \;\Big|\;
    a_m \in \R
    \Bigr\} \,.
\end{equation*}
Finally, we also include reconstructions obtained by the factorization method (see, e.g., \cite[Sec.~4]{KirGri08} or \cite[Sec.~7.5]{Kir21}) in the last column.
The corresponding plots show the logarithm of the reciprocal of the usual Picard sum.

At the wave number $k=0.5$ the monotonicity-based regularization gives good results for both noise levels as shown in Figures~\ref{fig:Exa_k0.5_delta0.01}--\ref{fig:Exa_k0.5_delta0.05}.
The plots of the linear constraint $\beta^*$ look very similar to the reconstructions obtained by the factorization method, which is not completely surprising since the theoretical foundations of the monotonicity test and of the factorization method are closely related.
The results obtained by Tikhonov regularization of the linearized problem are also good, which suggests that the linearization error is reasonably small.
This also partially explains why the regularized (by the sum of all positive eigenvalues of the linearized residual) minimization of the Frobenius norm of the linearized residual $\bfR^\delta(h)$ in \eqref{eq:RegTraceMinimization} works well.

At the larger wave number $k=1$ the linearization error increases,
accordingly also the results obtained by the monotonicity-based regularization and also by Tikhonov regularization of the linearized problem in Figures~\ref{fig:Exa_k1.0_delta0.01}--\ref{fig:Exa_k1.0_delta0.05} are not as good as for $k=0.5$.
However, the linear constraint $\beta^*$ still gives at least comparable results as the factorization method.
These results suggest that the monotonicity-based regularization method developed in this work is a good alternative to existing methods in the low frequency regime.
Moreover, implementing the monotonicity-based reconstruction schemes from \cite{GriHar18} using generalized eigenvalue problems as done here for the computation of the linear constraint $\beta^*$ seems to stabilize the method considerably.
\end{example}

\section*{Conclusions}
Based on \cite{EbeHarWan25}, we have proposed a new qualitative optimization scheme for shape reconstruction for inverse medium scattering problems.
This monotonicity-based regularization method combines the sound theoretical foundation of the monotonicity method with improved stability properties of a one-step linearization.
It works particularly well at low frequencies or for small contrasts in the refractive index, where the linearization error is small.
The method can be implemented efficiently using semidefinite programming and does not require solving any forward problems during reconstruction.
As a byproduct, we have developed a novel numerical implementation of the monotonicity test from \cite{GriHar18} by means of generalized eigenvalue problems, which has shown to give good numerical reconstructions also in the presence of noise.

\section*{Acknowledgments}
The research of R.~Griesmaier is supported by the Deutsche Forschungsgemeinschaft (DFG, German Research Foundation) - Project-ID 258734477 - SFB 1173.
The research of J.~Xiang is supported by the Natural Science Foundation of China (12301542) and the Open Research Fund of Hubei Key Laboratory of Mathematical Sciences (Central China Normal University, MPL2025ORG017).

{\small
  \bibliographystyle{abbrvurl}
  \bibliography{monotonicity_regularization}
}

\end{document}